\documentclass[11pt,reqno,a4paper]{amsart}

\usepackage[margin=1.15in]{geometry} % Balances page count and beautiful margins
\usepackage{microtype} % Typographical perfection, reduces awkward hyphenation
\usepackage{amsmath, amssymb, amsthm}
\usepackage{mathrsfs}
\usepackage[hidelinks]{hyperref}
\usepackage{esint}
\usepackage{mathtools}
\usepackage{graphicx}

\makeatletter
\def\@tocline#1#2#3#4#5#6#7{\relax
  \ifnum #1>\c@tocdepth %
  \else
    \par \addpenalty\@secpenalty\addvspace{#2}%
    \begingroup \hyphenpenalty\@M
    \@ifempty{#4}{%
      \@tempdima\csname r@tocindent\number#1\endcsname\relax
    }{%
      \@tempdima#4\relax
    }%
    \parindent\z@ \leftskip#3\relax \advance\leftskip\@tempdima\relax
    \rightskip\@pnumwidth plus4em \parfillskip-\@pnumwidth
    #5\leavevmode\hskip-\@tempdima
    \ifcase #1
     \or\or \hskip 1em \or \hskip 2em \else \hskip 3em \fi%
    #6\nobreak\relax
    \dotfill\hbox to\@pnumwidth{\@tocpagenum{#7}}\par
    \nobreak
    \endgroup
  \fi}
\makeatother

\newtheorem{theorem}{Theorem}[section]
\newtheorem{lemma}[theorem]{Lemma}

\newtheorem{corollary}[theorem]{Corollary}

\theoremstyle{definition}
\newtheorem{definition}[theorem]{Definition}
\newtheorem{example}[theorem]{Example}

\theoremstyle{remark}
\newtheorem{remark}[theorem]{Remark}

\DeclareMathOperator{\im}{im}
\newcommand{\R}{\mathbb{R}}
\newcommand{\C}{\mathbb{C}}
\newcommand{\A}{\mathbb{A}}
\newcommand{\Q}{\mathbb{Q}}
\newcommand{\B}{\mathbb{B}}
\newcommand{\defeq}{\stackrel{\textnormal{def}}{=}}

\DeclareMathOperator{\Tan}{Tan}

\begin{document}

\title[Rigidity and Dimensional Estimates in $BV^\A$]{Complex-ellipticity, dimensional estimates \\
and plane wave rigidity in $BV^\A$}

\author{Adolfo Arroyo-Rabasa}
\address{\newline 
Dipartimento di Matematica \newline
Universit\`a di Pisa \newline
 Largo Bruno Pontecorvo 5 \newline
  Pisa 56127, Italy \newline
Email: \textnormal{\texttt{\href{mailto:adolfo.rabasa@unipi.it}{adolfo.rabasa@unipi.it}}}}

\date{\today}


\begin{abstract}
We characterize complex-elliptic operators $\A(D)$ through a hierarchy of overdeterminacy ($\ell$-vanishing) quantifying the structural twisting of their symbols. This framework yields the optimal dimensional estimate for $BV^\A$-functions: a measure $\A u$ on $n$-dimensions cannot concentrate on sets of dimension below $n-1$. Consequently, the jump part of $\A u$ is characterized as an $(n-1)$-dimensional surface measure with density given by the symbol and the two-sided traces as it occurs for $BV$-functions. Building on this dimensional bound, we prove that measures satisfying $\frac{\A u}{|\A u|} \in \mathrm{span}\{P_0\}$ precisely decompose into finite sums of one-dimensional $BV$ profiles. Ultimately, these results reveal that complex-ellipticity strictly enforces a plane-wave structure on tangent measures.
\end{abstract}

\maketitle
\tableofcontents

\section{Introduction}

The study of the fine properties of functions and measures satisfying linear PDE constraints remains a constant source of fundamental questions, largely driven by its pivotal role in geometric measure theory and the calculus of variations~\cite{Alberti1993,AK_2000,BB2007,breit,DR2016,federer1969,VS2013}. A central challenge in this field is to establish sharp dimensional estimates: given a constant-coefficient operator $\A(D)$ on $\R^n$, if $\A(D) u = \mu$ for a Radon measure $\mu$, how singular can $\mu$ be, and what is the minimal Hausdorff dimension of the sets supporting its distribution?~\cite{AR2020proc,ADHR2019, AS2025, Ayoush2023,Dobronravov2024,SS2026,Stolyarov2023}

Within the class of elliptic operators, the gradient sets the gold standard. De Giorgi's foundational theory of sets of finite perimeter guarantees that a gradient measure $\mu = Du$ in $\R^n$ cannot concentrate on sets of dimension below $n-1$~\cite{AFP2000,federer1969}. In stark contrast, general elliptic operators lack this geometric rigidity; the fundamental solution of the Cauchy--Riemann equations, for instance, readily allows measures to concentrate on purely $0$-dimensional points. Establishing codimension-one estimates therefore requires the operator to be inherently more ``gradient-like'' or strictly \emph{overdetermined}. 

Unlike square systems, strongly overdetermined systems impose internal algebraic compatibilities that obstruct lower-dimensional solutions. We capture this through a precise geometric rigidity property: an operator $\A(D)$ must force the intersection of its image spaces across any $\ell$-dimensional frequency plane to be strictly trivial. Specifically, its principal symbol $\A(\zeta)$ must satisfy the \emph{$\ell$-vanishing} condition (see Def.~\ref{def:ell} for the formal setup):
\[
    \bigcap_{V \in \mathrm{Gr}(n,\ell)} \,  \bigcup_{\zeta \in V} \im \A(\zeta) = \{0\},
\]
where $\mathrm{Gr}(n,\ell)$ denotes the Grassmannian manifold of $\ell$-planes in $\R^n$. This geometric property quantifies overdeterminacy not merely by a surplus of equations, but by the structural ``twisting'' of the operator's symbol across frequency space. We demonstrate that, within the class of elliptic operators, \emph{complex-ellipticity} (see Definition~\ref{def:CE}) represents the sharpest and most extreme manifestation of this geometric twisting (see Theorem~\ref{thm:main}).

Uncovering this precise geometric rigidity unlocks a refined understanding of the fine properties of $BV^\A$-functions. Most notably, it ensures through the results in~\cite{ADHR2019} that the associated measures cannot concentrate on lower-dimensional sets, establishing the optimal codimension-one dimensional estimate $|\A(D) u| \ll \mathscr{H}^{n-1}$ (see Corollary~\ref{cor:estimates}). This dimensional bound, in turn, allows us to improve upon the known theory of interior jump traces. While previous representation formulas for the singular measure were strictly confined to countably $\mathscr{H}^{n-1}$-rectifiable sets (i.e., countable unions of Lipschitz graphs)~\cite{AS2025}, our estimates rigorously extend the existence of Lebesgue-pointwise two-sided traces to arbitrary $\mathscr{H}^{n-1}$-rectifiable sets (see Corollary~\ref{cor:traces}). In particular, for first-order complex-elliptic operators, this guarantees that the lowest-dimensional singular component of $\A u$ is completely characterized by the jump part (see Corollary~\ref{cor:jump}):
\begin{align*}
    \A u \lfloor \{  \theta^{*(n-1)}(|\A u|) > 0\} & = \A^j u  \defeq \A u  \lfloor_{J_u} \\ 
    & = \A(\nu)(u^+ - u^-) \, \mathscr H^{n-1} \lfloor_{J_u},
\end{align*}
where $J_u$ is the jump set of $u$, $\nu$ is a Borel measure-theoretic normal to $J_u$, and $u^+, u^-$ are the corresponding two-sided traces of $u$ on $J_u$. 

Beyond fine properties, this structural rigidity has profound implications for the calculus of variations, particularly in the study of integral functionals with linear growth defined on $BV^\A$-spaces. The ensuing \emph{plane-wave decompositions} of tangent measures (see Lemma~\ref{lem:rigidity} and Corollary~\ref{cor:poly_tangents}) allow such variational problems to be locally reduced to purely one-dimensional profiles. Such structure theorems help bypass the need for Alberti's rank-one theorem~\cite{KR2010,Rindler2012}, but crucially are also fundamental for integral representation results via homogenization, and understanding the singular sets of strain fields in perfect plasticity and fracture mechanics~\cite{CVG2026,DPR2020}---where the relevant operators, such as the symmetric and deviatoric gradients, are precisely complex-elliptic.

\section{Set-up}
To properly frame this class and our main results, we first fix the underlying notation. Let $\Omega \subset \R^n$ be an open set and let $E,F$ be finite-dimensional inner product $\R$-vector spaces with dimensions $M,N \ge 1$ and bases $e_1,\dots,e_M$ and $f_1,\dots,f_N$ respectively. A linear homogeneous system $\A(D) :\mathscr D'(\R^n,E) \to \mathscr D'(\R^n,F)$ of order $k \ge 1$ acts on an $E$-valued distribution as
\begin{equation}\label{eq:A}
    \A u \coloneqq  \A(D)u  \, = \, \sum_{\substack{s=1,\dots,M\\j=1,\dots,N}} (a_{js}(D)\, u_s)f_j 
\end{equation}
where $u =u_1e_1 + \dots + u_Me_M$ and each $a_{js}(D)$ is a real-valued, $k$-th order homogeneous operator with constant coefficients $a_{js} \in \R[x_1,\dots,x_n]$. The principal symbol associated with $\A(D)$ is the polynomial map $\A : \R^n \to  \mathcal L(E,F)$ defined by $\A(\zeta) = (a_{js}(\zeta))$, where $\zeta \in \R^n$. Alternatively, it is common to write $\A(D)$ in (higher-order) divergence form:
\[
    \A u \,  = \sum_{|\alpha|=k} A_\alpha \partial_\alpha u,  \qquad A_\alpha \in \mathcal L(E,F),
\]
in which case
\[
    \A(\zeta) = \sum_{|\alpha| = k} A_\alpha \zeta^\alpha, \qquad  \zeta^\alpha \defeq \zeta_1^{\alpha_1} \cdot \cdots \cdot \zeta_n^{\alpha_n}.
\]

In order to define complex-ellipticity, we interchangeably consider $\A$ as a polynomial with real coefficients on $\C^n$. Throughout, we write $\zeta$ to denote vectors in $\R^n$ and reserve $\xi$ to denote complex vectors in $\C^n$. 

\begin{definition}[$\mathbb{K}$-elliptic]\label{def:CE} Let $\mathbb{K} \in \{\R,\C\}$.
An operator $\A(D)$ as in $\eqref{eq:A}$ is called \emph{$\mathbb{K}$-elliptic} if the family of all $M \times M$ sub-determinants (minors) of the symbol matrix $\A(\xi)$ has no common zero in $\mathbb{K}^n \smallsetminus \{0\}$. That is, for every frequency $\xi \in \mathbb{K}^n \smallsetminus \{0\}$, the map $\A(\xi) : \mathbb{K} \otimes E \to \mathbb{K} \otimes F$ has full column $\mathbb{K}$-rank $M$.
\end{definition} 

Notice that this definition inherently imposes a form of overdeterminacy ($M < N$) when $\mathbb{K} = \C$ due to the algebraic closure of the complex field. By homogeneity and the rank-nullity theorem, $\mathbb{K}$-ellipticity is strictly equivalent to the following coercivity property: there exists $c_\A > 0$ such that
\[
    |\A(\xi) e| \ge c_\A |\xi|^k|e| \quad  \text{for all } \xi \in \mathbb{K}^n \text{ and } e \in \mathbb{K} \otimes E.
\]

Smith~\cite{Smith61,Smith70} identified $\C$-ellipticity as the \emph{sufficient and necessary} condition behind the validity of coercive $L^p$-inequalities on domains, i.e., $\|Du\|_{p,\Omega} \lesssim \|u\|_{p,\Omega} + \|\A u\|_{p,\Omega}$ for $1 < p < \infty$. Furthermore, Breit, Diening, and Gmeineder \cite{breit} proved that $\C$-ellipticity is the \emph{unique} condition guaranteeing a bounded exterior trace $\| \mathrm{tr}_{\partial \Omega} u\|_{1,\partial \Omega} \lesssim \|u\|_{1,\Omega} + \|\A u\|_{1,\Omega}$---subsequently, it was demonstrated by the author and Skorobogatova that this exterior trace behaves precisely as the classical pointwise trace in the sense of Lebesgue~\cite{AS2025}. More recently, Borken's~\cite{Borken2025} outstanding generalization of the quantitative estimates by Friesecke, James, and M\"uller~\cite{FJM02} (see also~\cite{LLP2024}) revealed that $\C$-ellipticity plays a key structural role in the mechanisms governing the geometric rigidity of $\mathrm{SO}(n)$.

\medskip

In light of this robust analytic framework, it has been a long-standing conjecture that $\C$-elliptic constraints enforce $(n-1)$-dimensional rigidity (see, e.g., \cite{AS2025,AW2017,Raita}). We bridge this gap by characterizing $\C$-ellipticity through a hierarchy of geometric overdeterminacy, inspired by the \emph{wave} and \emph{image} cones introduced in~\cite{AR2021,ADHR2019}. 

\begin{definition}[Vanishing] \label{def:ell} Let $\ell \in \{1,\dots,n-1\}$ and let $\mathrm{Gr}(n,\ell)$ denote the Grassmannian manifold of $\ell$-subspaces of $\R^n$. An operator $\A(D)$ as in~\eqref{eq:A} is called \emph{vanishing on $\ell$-planes} (or simply \emph{$\ell$-vanishing}) if  its $\ell$-dimensional image cone is trivial, i.e.,  
\begin{equation}\label{eq:bundle}
    \mathcal I^{\ell}_\A \defeq \bigcap_{V \in \mathrm{Gr}(n,\ell)} \, \bigcup_{\zeta \in V}  \im \A(\zeta)  = \{0_F\}.
\end{equation}
\end{definition}

It is immediate to demonstrate that these cones define a geometric hierarchy
\[
    \mathcal I^1_\A \subset \cdots \subset \mathcal I^\ell_\A \subset \mathcal I^{\ell+1}_\A \subset \cdots \subset \mathcal I^{n-1}_\A.
\]
One can interpret this as a scale quantifying the geometric `twistedness' of an operator's image bundle; the higher the $\ell$, the stronger the vanishing condition. The lowest tier corresponding to lines, \emph{$1$-vanishing}, coincides precisely with Van Schaftingen's cancellation~\cite{VS2013}:
\[
	\bigcap_{\zeta \in \R^n \smallsetminus \{0\}} \im \A(\zeta) = \{0\},
\] 
which is a condition underpinning the endpoint Sobolev estimates of Bourgain and Brezis~\cite{BB2004, BB2007}. In general, $\ell$-vanishing is stronger than cancellation on $(\ell -1)$-planes:
\[
	\bigcap_{\zeta \in V \smallsetminus \{0\}} \im \A(\zeta) \subset \mathcal I^\ell_\A  \qquad \text{for all $V \in \mathrm{Gr}(n,\ell-1)$.}
\]
It is, however, weaker than \emph{strong $(\ell-1)$-cancellation}~\cite{SV2019}, which is equivalent to the admissibility of $\ell$-dimensional slicing for the associated measures~\cite{AR2020}.

\medskip

Our main algebraic result characterizes first-order complex-elliptic operators as the endpoint of the hierarchy introduced above, representing the maximal degree of structural torsion:
\begin{theorem}\label{thm:main} 
Let $\A(D)$ be a first-order operator as in~\eqref{eq:A}. The following are equivalent:
\begin{enumerate}
    \item[(i)] $\A(D)$ is $\C$-elliptic;
    \item[(ii)] $\A(D)$ is $\R$-elliptic and vanishing on hyperplanes;
    \item[(iii)] $\A(D)$ is $\R$-elliptic and there exists an integer $r \ge 1$ depending only on $\A$ such that    
    \[
    	\#  \{\zeta \in S^{n-1} : f \in  \im \A(\zeta) \} \le r \qquad \forall f \in F.
    \]
\end{enumerate} 
\end{theorem}

Theorem \ref{thm:main} reveals that complex-ellipticity is not an isolated algebraic condition, but the endpoint of a geometric scale. Every $\R$-elliptic and canceling operator sits somewhere on this spectrum: bounded below by $1$-vanishing (yielding endpoint Sobolev bounds), all the way to complex-ellipticity, where the bundle's twisting becomes so severe that the associated measures $\A u$ are forced to exhibit codimension-one rigidity (see below).

\medskip

A simple order-reduction principle~\cite{AS2025} (see Sect.~\ref{sec:higher}), preserving complex-ellipticity, transforms $\A(D)$ into a first-order operator $\mathbb{L}_\A(D)$ and structurally satisfies $\mathcal I^{n-1}_\A \subset \mathcal I^{n-1}_{\mathbb{L}_\A}$, therefore allowing us to identify the bundle property as a strictly \emph{necessary} condition for general complex-elliptic operators:

\begin{corollary}\label{cor:high1}
Every $\C$-elliptic operator is hyperplane vanishing.
\end{corollary}

For operators of higher order $(k \ge 2)$, the equivalence in Theorem~\ref{thm:main} may fail for vectorial operators ($M > 1$); in fact, it fails for $k = 2$ and $M=4$ (cf. Example~\ref{ex}). For higher-order operators, the exact geometric equivalence is captured by the hyperplane vanishing of the associated first-order reduction:

\begin{corollary}\label{cor:high}
    The following are equivalent:
\begin{enumerate}
    \item[(i)] $\A(D)$ is $\C$-elliptic;
    \item[(ii)] $\mathbb{L}_\A(D)$ is $\R$-elliptic and vanishing on hyperplanes.
\end{enumerate} 
\end{corollary} 

\section{Applications}
To properly frame our dimensional estimates and rigidity results, we introduce the space of functions with bounded $\A$-variation. For an open set $\Omega \subset \R^n$, we define
\[
    BV^\A(\Omega, E) \defeq \big\{ u \in L^1(\Omega, E) : \A(D)u \in \mathcal{M}(\Omega, F) \big\},
\]
where $\mathcal{M}(\Omega, F)$ denotes the space of $F$-valued finite Radon measures on $\Omega$. This space is naturally equipped with the norm $\|u\|_{L^1(\Omega)} + |\A(D)u|(\Omega)$. We write $BV^\A_{loc}(\Omega, E)$ when the property holds on every relatively compact open subset.

\subsection{Fine properties}With the strict triviality of the wave cone established, the codimension-$1$ dimensional estimates follow as a direct consequence of the structure theorems for PDE-constrained measures studied in~\cite{ADHR2019}. Let us recall that the \emph{$\ell$-dimensional integral geometric measure} of a Borel set $A \subset \R^n$ is (see, e.g.,~\cite{federer1969,Mat}):
\[
    \mathscr I^\ell(A) \, \defeq \int_{\mathrm{Gr}(n,\ell)} \left( \int_V \mathscr H^0(A \cap P_V^{-1}\{a\}) \, d\mathscr H^\ell(a) \right) d\gamma_{n,\ell}(V),
\]
where $\mathrm{Gr}(n,\ell)$ is the Grassmannian manifold of $\ell$-dimensional subspaces of $\R^n$, endowed with its unique Haar measure $\gamma_{n,\ell}$ inherited from $\mathrm{O}(n,n-1)$, and $P_V$ denotes the canonical projection onto $V$. 

\begin{remark}
The integral geometric measure is naturally dominated by its Hausdorff counterpart: $\mathscr I^\ell(A) \le \mathscr H^\ell(A)$. Equality holds when $A$ is $\mathscr H^\ell$-rectifiable thanks to the Besicovitch--Federer projection theorem.
\end{remark}
 
We are now in a position to state the dimensional estimate:

\begin{corollary}\label{cor:estimates} 
Let $\A(D)$ be a complex-elliptic operator. For any $u \in BV^\A_{loc}(\Omega, E)$, its associated measure satisfies
\[
    |\A(D) u| \ll \mathscr I^{n-1} \ll \mathscr H^{n-1}.
\]
\end{corollary} 

This bound, combined with the methodologies developed in~\cite{AS2025}, yields a subtle, yet nontrivial, improvement regarding the explicit representation of measures $\A(D) u$ when restricted to rectifiable sets.

The setting is as follows: let $\Gamma \subset \Omega$ be an $\mathscr H^{n-1}$-rectifiable set equipped with a Borel measure-theoretic unit normal $\nu : \Gamma \to S_E$, where $S_E$ denotes the unit sphere in $E$. Following~\cite{AS2025}, we recall that any $u \in L^1_{\mathrm{loc}}(\Omega,E)$ for which $\A(D) u$ is represented by an $F$-valued Radon measure admits well-defined two-sided traces $u^+(x),u^-(x) \in E$ for $\mathscr H^{n-1}$-almost every $x \in \Gamma$. These traces satisfy one-sided Lebesgue point limits
\[
     \lim_{r \to 0} \fint_{B^\pm(x,r;\nu)} |u^\pm (x) - u(y)| \, dy = 0,
\]
where 
\[
B^\pm(x,r;\nu) = \{ y \in B(x,r) : \pm(y-x) \cdot \nu(x) > 0\}.
\] 

In~\cite{AS2025}, it was shown that if $\A(D)$ is complex-elliptic and $\Gamma$ is the image of a Lipschitz subset of $\R^{n-1}$ (and, by extension, on countably $\mathscr H^{n-1}$-rectifiable sets), then the restriction of $\A u$ to $\Gamma$ can be expressed as a surface measure with density $\A(\nu)(u^+ - u^-)$. 
We can now extend this representation to arbitrary $\mathscr H^{n-1}$-rectifiable sets.

\begin{corollary}\label{cor:traces}
Let $\A(D)$ be a first-order $\C$-elliptic operator and let $\Gamma \subset \Omega$ be an $\mathscr H^{n-1}$-rectifiable set with a Borel measure-theoretic normal $\nu : \Gamma \to S_E$. If $u \in BV^\A_{loc}(\Omega, E)$, then
\[
         \A u \, \lfloor_{\, \Gamma} \,  = \A(\nu)(u^+ - u^-) \, \mathscr H^{n-1} \, \lfloor_{\, \Gamma \cap J_u} \quad \text{as measures on } \Omega.
\]
In particular, the two-sided trace satisfies the bound
\[
    \| u^+ - u^- \|_{L^1(\Gamma; \mathscr H^{n-1})} \le c_\A^{-1} |\A u|(\Gamma).
\]
\end{corollary}

\begin{remark}
A similar statement holds when $\A(D)$ is of higher order. The details, which can be retrieved directly from the proof of Corollary~\ref{cor:traces} and~\cite[Sect.~2.5]{AS2025}, are left to the reader. 
\end{remark}

To formally isolate the lowest-dimensional singular component of the measure, we recall the definition of the upper $(n-1)$-dimensional Hausdorff density. For a positive Radon measure $\mu \in \mathcal{M}(\Omega)$ and a point $x \in \Omega$, it is given by
\[
    \theta^{*(n-1)}(\mu, x) \defeq \limsup_{r \to 0^+} \frac{\mu(B_r(x))}{\omega_{n-1} r^{n-1}},
\]
where $\omega_{n-1}$ is the volume of the $(n-1)$-dimensional unit ball in $\R^{n-1}$. 

\begin{corollary}[Jump part characterization]\label{cor:jump}
    Let $\A(D)$ be a first-order $\C$-elliptic operator and let $u \in BV^\A(\Omega, E)$. Defining the jump measure as the restriction
    \[
        \A^j u \defeq \A u \lfloor \{x \in \Omega : \theta^{*(n-1)}(|\A u|, x) > 0\},
    \]
    and letting $\nu : J_u \to S_E$ be a Borel measure-theoretic normal of the jump set $J_u$, we have
    \[
        \A^j u = \A(\nu)(u^+ - u^-) \, \mathscr H^{n-1} \lfloor_{J_u},
    \]
    where $u^+,u^-$ are the two-sided Borel traces of $u$ on $J_u$ with respect to $\nu$.
\end{corollary}

\subsection{Plane-wave rigidity} Plane-wave decompositions and their profound connection to the structural rigidity of singular measures have become a cornerstone in the calculus of variations. The crucial role of such decompositions was first highlighted in the characterization of lower semicontinuous envelopes for integral functionals in $BV$ and $BD$ without relying on Alberti's rank-one theorem, as pioneered by Kristensen and Rindler~\cite{KR2010} and Rindler~\cite{Rindler2011,Rindler2012}. Building upon these foundational insights, explicit plane-wave decompositions were subsequently established for functions of bounded deformation ($BD$) by De Philippis and Rindler~\cite{DPR2020}, and more recently extended to the maps with bounded deviatoric strain ($\mathrm{BD}^{\text{dev}}$) by Caroccia and Van Goethem~\cite{CVG2026}. Before stating our general rigidity results for $\C$-elliptic operators, we formally define the functional classes of plane waves that naturally arise from constant-coefficient annihilators.

\begin{definition}[Plane waves]\label{def:plane_waves}
Let $\Omega \subset \R^n$ be an open set. To rigorously define distributions with one-dimensional profiles, we fix a direction $\zeta \in S^{n-1}$ and choose an orthogonal rotation $R_\zeta \in \mathrm{O}(n)$ such that $R_\zeta e_1 = \zeta$. This rotation induces a coordinate splitting $\R^n \cong \R \times \R^{n-1}$ where any point $x \in \R^n$ is uniquely expressed as $x = R_\zeta(t, y)$ with $t = x \cdot \zeta \in \R$ and $y \in \R^{n-1}$. We denote the corresponding rotated domain by $\Omega_\zeta \defeq R_\zeta^{-1}(\Omega)$. Naturally, if $\Omega = B_\rho$ is a ball centered at the origin, the domain remains invariant under this rotation, yielding $\Omega_\zeta = B_\rho$.

Given a \emph{profile} 1D distribution $g \in \mathscr D'(\R, E)$ and a scalar distribution $h \in \mathscr D'(\R^{n-1})$, their tensor product $g \otimes h$ is canonically defined on the split space $\R \times \R^{n-1}$. A distribution $u \in \mathscr D'(\Omega, E)$ is called a \emph{plane wave} if there exists a direction $\zeta \in S^{n-1}$ and a 1D distribution (profile) $g \in \mathscr D'(\R, E)$ such that $u$ is the pushforward from the rotated domain:
\[
    u  = (R_\zeta)_\# (g \otimes \mathscr L^{n-1}) \lfloor_\Omega.  
\]
Equivalently, the pullback distribution $R_\zeta^* u$ coincides with $g \otimes \mathscr L^{n-1}$ exactly on $\Omega_\zeta$. 

We say $u$ is a \emph{polynomial plane wave} of degree $M$ if it can be written as a finite sum
\[
    u = \sum_{\beta} (R_\zeta)_\# \big(g_\beta \otimes (q_\beta \mathscr L^{n-1})\big) \lfloor_\Omega,
\]
where $q_\beta$ are scalar polynomials on $\R^{n-1}$ of degree at most $M$, and $g_\beta \in \mathscr D'(\R, E)$. For brevity, when such a distribution is represented by a regular $L^1_{loc}(\Omega)$ function, we will formally denote it as $u(x) = \sum_\beta q_\beta(y_{\zeta}) g_\beta(x \cdot \zeta)$, where $y_\zeta \in \R^{n-1}$ are the corresponding transverse coordinates to $\zeta$.
\end{definition}

\begin{lemma}[Rigidity]\label{lem:rigidity} 
	Let $\A(D)$ be a first-order complex elliptic operator as in~\eqref{eq:A} and let $\Omega \subset \R^n$ be an open and convex set. There exist integers $m,r \ge 1$ depending only on  the symbol $\A(\zeta)$ with the following property: if $u \in \mathscr D'(\Omega,E)$ satisfies
	\begin{equation}\label{eq:rig}
		\A(D) u = P \sigma  
	\end{equation}
	for some vector $P \in F$ and some signed measure $\sigma \in \mathcal M(\Omega)$, 
	then $\sigma$ and $u$ can be decomposed into a sum of at most $r$ transversally polynomial plane waves with distinct characteristic directions $\zeta_1,\dots,\zeta_r \in S^{n-1}$. Specifically:
	\[
	\sigma(x) = \sum_{j=1}^r \sigma_j(x), \qquad u(x) = Q_{m-1}(x) + \sum_{j=1}^r u_j(x),
	\]
    where $Q_{m-1}$ is a polynomial of degree at most $m-1$, each $\sigma_j$ is a scalar transversally polynomial plane wave measure, and each $u_j$ is an $E$-valued transversally polynomial plane wave with $BV_{loc}$ profiles along $\zeta_j$.
\end{lemma} 

\begin{corollary}
    Let $\Omega \subset \R^n$ be an open and convex set and let $u \in \mathscr D'(\Omega,E)$ satisfy the differential constraint
    \begin{equation*}
		\A(D) u = P \sigma  
	\end{equation*}
	in the sense of distributions for some vector $P \in F$ and some signed (finite) measure $\sigma \in \mathcal M(\Omega)$. Then $u \in BV(\Omega,E)$. 
\end{corollary}

We write $B_\rho(x)$ to denote the open ball with radius $\rho$ centered at $x \in \R^n$. For simplicity, we write $B_\rho = B_\rho(0)$. Let $u \in BV^\A(\Omega, E)$ and let $x \in \Omega$. For a given $\rho > 0$, we define the $\rho$-blow up of $u$ at $x$ at scale $r>0$ as
\[
	v_{x,\rho}(y) \defeq    u(x + \rho y) - \fint_{B_\rho(x)} u, \qquad y \in B_1.
\]
A map $v : B_1(0) \to E$ is called a $BV^\A$-tangent of $u$ at $x$ provided that it is an accumulation point of the renormalized blow-up family $\{c_{x,\rho} \cdot v_{x,\rho}\}$ with respect to the $L^1$-topology, i.e., there exists an infinitesimal sequence $\rho_j \to 0^+$  such that
\begin{equation}\label{eq:tan}
	c_{x,\rho_j} \cdot v_{x,\rho_j} \longrightarrow v \; \text{in } L^1(B_1,E), \qquad c_{x,\rho_j} \coloneqq \frac{1}{|\A u|(B_{\rho_j}(x))}.
\end{equation}

Following standard notation for $BV$ maps, we write $\Tan_{BV^\A}(u,x)$ to denote the set of all $BV^\A$-tangents of $u$ at $x$. The theory developed in~\cite{AS2025} and the theory of tangent measures (e.g,~\cite{AFP2000}) guarantees that the tangent cone is non-zero, i.e., $\Tan_{BV^\A}(u,x) \neq \{0\}$ for $|\A u|$-almost every $x \in \Omega$. Similarly, Lebesgue's theorem yields that
\[
	\A v(dy) = \frac{\A u}{|\A u|}(x) \, |\A v|(dy)  \qquad \text{for all $v \in \Tan_{BV^\A}(u,x)$}
\]
for every $x$ in a full $|\A u|$-measure subset of $\Omega$. 

As a direct consequence of the plane-wave decomposition (Lemma~\ref{lem:rigidity}) and the structure of $BV^\A$-tangents, we obtain the following description, which is often useful in the calculus of variations:

\begin{corollary}\label{cor:poly_tangents}
	Let $\A(D)$ be a complex elliptic operator and let $u$ be a map in $BV^\A(\Omega,E)$. There exists a full $|\A u|$-measure set $\Omega' \subset \Omega$ such that for every $x \in \Omega'$, any tangent $v \in \mathrm{Tan}_{BV^\A}(u,x)$ is a finite sum of transversally polynomial plane waves taking the form:
	\[
		v(y) = Q_{m-1}(y) + \sum_{j=1}^r \sum_{|\beta| \le m} q_{j,\beta}(y_{\zeta_j}) \, g_{j,\beta}(y \cdot \zeta_j),
	\]
    where $y_{\zeta_j} \in \R^{n-1}$ are the respective transverse coordinates and each 1D profile $g_{j,\beta} \in BV_{loc}(\R,E)$.
\end{corollary}

Since tangents of tangent maps are themselves tangent maps, an iterated blow-up argument allows us to eliminate the transversal polynomials entirely, geometrically isolating a pure plane wave:

\begin{corollary}[Pure Plane Wave Tangents]\label{cor:pure_tangents}
    Let $\A(D)$ be a complex elliptic operator and let $u \in BV^\A(\Omega,E)$. There exists a full $|\A u|$-measure set $\Omega' \subset \Omega$ such that for every $x \in \Omega'$, the set of tangents $\mathrm{Tan}_{BV^\A}(u,x)$ contains a pure plane wave map:
    \[
        v(y) = B \, g(y \cdot \zeta_0)
    \]
    for some direction $\zeta_0 \in S^{n-1}$, some vector $B \in E$ and some non-constant 1D profile $g \in BV_{loc}(\R)$.
\end{corollary}

\begin{remark}[Preiss' tangents] Similar results hold if one considers globally defined tangents by requiring 
	\[
		c_j \cdot v_{x,\rho_j} \longrightarrow v \; \text{in } L^1_\mathrm{loc}(\R^n,E), 
	\]
for some nonzero map $v$ and some sequence $(c_j)$ of positive constants.  
\end{remark}

\section{Auxiliary lemmas}\label{pre}
Before proceeding to the proofs of our main results, we establish the core geometric mechanism that translates the algebraic condition of hyperplane vanishing into the existence of a continuous manifold of frequencies. The central idea is that if the image bundle of an operator twists severely enough to vanish on all hyperplanes, the set of directions mapping to a common target vector must be sufficiently ``large'' to contain a smooth curve. To rigorously quantify this size, we first prove a spherical co-area inequality (Lemma~\ref{lem:co}) that bounds the one-dimensional Hausdorff measure of spherical sets by their average number of intersections with great equators. We then apply this integral-geometric tool to the real algebraic variety generated by the operator's symbol (Lemma~\ref{lem:crofton}). This ensures the existence of a non-trivial smooth curve, which will serve as the topological seed for the complex-analytic rigidity argument in the subsequent section.

\begin{lemma}[Crofton-type formula]\label{lem:co}
Let $A \subset S^{n-1}$ be a Borel set. Defining 
\[
    \mathscr{I}_{\mathrm{sph}}^1(A) \defeq \int_{\mathrm{Gr}(n, n-1)} \mathscr{H}^0(A \cap H) \, d\gamma_{n,n-1}(H),
\]
where $\gamma_{n,n-1}$ is the unique Haar measure on the Grassmannian, we have 
\[
    \mathscr{I}_{\mathrm{sph}}^1(A) \le c_n \mathscr{H}^1(A),
\]
for some dimensional constant $c_n > 0$.
\end{lemma}

\begin{proof} For simplicity, we write $G = \mathrm{Gr}(n,n-1)$ and $\gamma = \gamma_{n,n-1}$.
Let $\delta > 0$. By the definition of the $1$-dimensional Hausdorff measure, there exists a countable covering $\{C_i\}_{i=1}^{\infty}$ of $A$ by spherical sets such that $\mathrm{diam}(C_i) \le \delta$ for all $i$, and
\[
    \sum_{i=1}^{\infty} \mathrm{diam}(C_i) \le \mathscr{H}_\delta^1(A) + \delta.
\]
For any fixed set $C_i$ on the sphere, the $\gamma$-measure of the set of great equators intersecting $C_i$ is bounded proportionally by its diameter. Thus, with the appropriate constant $c_n$, we have
\begin{equation}\label{eq:diam}
 \int_{G} \mathbf{1}_{\{C_i \cap H \neq \emptyset\}} \, d\gamma(H) \le c_n \mathrm{diam}(C_i).
\end{equation}
Now, consider a fixed equator $H \in G$. If $H$ intersects $A$ on at least $k$ distinct points, there exists a strictly positive minimum pairwise distance between these intersection points. Consequently, for sufficiently small $\delta > 0$, any covering set $C_i$ with $\mathrm{diam}(C_i) \le \delta$ can contain at most one of these intersection points. Therefore, the equator $H$ must intersect at least $k$ distinct covering sets:
\[
    k \le \liminf_{\delta \to 0} \sum_{i=1}^{\infty} \mathbf{1}_{\{C_i \cap H \neq \emptyset\}}.
\]
Taking the supremum over the number of intersection points, we obtain $\mathscr{H}^0(A \cap H) \le \liminf_{\delta \to 0} \sum_{i=1}^{\infty} \mathbf{1}_{\{C_i \cap H \neq \emptyset\}}$. Integrating over $G$ and applying Fatou's Lemma, we find:
\begin{align*}
\int_{G} \mathscr{H}^0(A \cap H) \, d\gamma(H) &\le \liminf_{\delta \to 0} \sum_{i=1}^{\infty} \int_{G} \mathbf{1}_{\{C_i \cap H \neq \emptyset\}} \, d\gamma(H) \\
&\le c_n \liminf_{\delta \to 0} \sum_{i=1}^{\infty} \mathrm{diam}(C_i) \le c_n (\mathscr{H}^1(A) + \delta).
\end{align*}
Taking $\delta \to 0$ yields the result.
\end{proof}

Equipped with this spherical co-area inequality, we can now prove that hyperplane vanishing rigorously enforces the existence of non-trivial curves within the operator's image variety.

\begin{lemma}\label{lem:crofton} 
Let $n \ge 2$ and let $\A(D)$ be an operator as in~\eqref{eq:A}. Suppose that $w \in \mathcal I^{n-1}_\A$. Then, the punctured space
\[
    V_w = \{ \zeta \in \R^n \smallsetminus \{0\} : w \in \im \A(\zeta) \}
\]
contains a non-trivial smooth projective curve.
\end{lemma}

\begin{proof} 
Let $\overline{V_w} = V_w \cup \{0\}$. Because a vector $w$ lies in the column space of $\A(\zeta)$ if and only if the $(M+1)$-vector wedge product of the columns of $\A(\zeta)$ with $w$ vanishes, $\overline{V_w}$ is exactly the real algebraic variety defined by the simultaneous vanishing of all $(M+1) \times (M+1)$ minors of the augmented matrix $[\A(\zeta)\mid w]$. Since the symbol $\A(\zeta)$ is homogeneous, $\overline{V_w}$ is a conical variety.

Consider the spherical slice $A \defeq V_w \cap S^{n-1}$, which is a real algebraic variety on the sphere. By hypothesis, $A$ intersects every equator $H = V \cap S^{n-1}$. Thus, $\mathscr{H}^0(A \cap H) \ge 1$ for all $H \in G$. Integrating this inequality, the spherical co-area inequality (Lemma~\ref{lem:co}) implies:
\[
    0 < \gamma(G) \le \int_G \mathscr{H}^0(A \cap H) d\gamma \le c_n \mathscr{H}^1(A),
\]
hence $\mathscr{H}^1(A) > 0$. As $A$ is a real algebraic variety with strictly positive $1$-dimensional Hausdorff measure, its topological dimension must be at least $1$. By the Whitney stratification theorem (see, e.g., Bochnak, Coste, and Roy~\cite{BCR98}), $A$ admits a stratification into a finite disjoint union of smooth manifolds $S_0 \cup S_1 \cup \cdots \cup S_d$. Since $\mathscr{H}^1(A) > 0$, $A$ must contain a $1$-dimensional stratum. Thus, $A$ contains a non-trivial smooth regular curve $\gamma \subset S^{n-1} \cap V_w$. This completes the proof.
\end{proof}

\section{Proofs of the main results}

In all that follows we may assume that $n\ge 2$, for otherwise all the results presented here follow directly from the triviality of one-dimensional operators.

Before proceeding to the proof of the main theorem, it is worth spending a few words on the ideas driving the equivalence in Theorem~\ref{thm:main}. The true novelty here is the fact that complex-ellipticity forces vanishing on hyperplanes. This is deduced by combining the geometric intersections of Lemma~\ref{lem:crofton} with the rigidity of holomorphic maps of several variables (Hartogs' theorem). Conversely, the reverse implication rests on an underlying algebraic rigidity: any first-order elliptic operator failing to be complex-elliptic must inherently contain a $2$-dimensional real plane where it behaves exactly like the Cauchy-Riemann equations. Since every $(n-1)$-dimensional hyperplane in $\R^n$ necessarily intersects this $2$-plane, a common non-zero vector must exist in the intersection, forcing hyperplane vanishing.

\begin{proof}[Proof of Theorem \ref{thm:main}]
We organize the proof into three steps establishing the logical cycle (i) $\implies$ (iii) $\implies$ (ii) $\implies$ (i).

\medskip

\emph{Step 1: (i) $\implies$ (iii)} \\
Assume $\A(D)$ is first-order and complex-elliptic. We first establish that for any non-zero $f \in F$, the complex projective variety associated with the punctured cone
\[
    V_f = \{ \xi \in \C^n \smallsetminus \{0\} : f \in \im \A(\xi) \}
\]
must have dimension zero. Suppose for contradiction that its projective dimension is at least $1$. Then $V_f$ contains a $1$-dimensional complex projective curve $\Gamma \subset \mathbb{CP}^{n-1}$, which generates a complex cone $\mathscr C \subset \C^n$ of complex dimension $2$. 

We claim that the punctured cone $\mathscr C^* \defeq \mathscr C \smallsetminus \{0\}$ is connected. 
By construction, $\mathscr C^*$ is given by the union  $\bigcup_{\zeta \in \gamma} L_\zeta$ where $L_\zeta = \{\lambda \zeta : \lambda \in \C \smallsetminus \{0\}\}$ is the punctured complex line generated by a non-zero real vector $\zeta \in \R^n$. Because $L_\zeta$ is topologically isomorphic to $\C$, it is connected. Now, let $x = \lambda \zeta_1, y=\eta \zeta_2 \in \mathscr C^*$. By the observation above, there is a continuous path from $x$ to $\zeta_1$ and from $y$ to $\zeta_2$. Since $\gamma$ is also connected, then there is also a continuous path from $\zeta_1$ to $\zeta_2$. This proves that $\mathscr C^*$ is connected.

Because $\A(D)$ is complex-elliptic, $\ker_\C \A(\xi) = \{0\}$ for all $\xi \ne 0$. Therefore, for each $\xi \in \mathscr C^*$, the linear system $\A(\xi) v = f$ admits a unique solution $v(\xi) \in \C \otimes E$. By Cramer's rule, the map $\xi \mapsto v(\xi)$ is locally rational, and given the connectedness of $\mathscr C^*$, it defines a vectorial holomorphic function on $\mathscr C^*$. The $1$-homogeneity of the symbol dictates $v(\lambda \xi) = \lambda^{-1} v(\xi)$ for all $\lambda \in \C \smallsetminus \{0\}$. By Hartogs' extension theorem (see, e.g.,~\cite[Thm.~2.6.7]{JP2000}), since $\mathscr C^*$ has complex dimension $2$, $v$ must extend analytically to the origin, implying it is bounded near $0$. This contradicts the blow-up singularity $|v(\lambda \xi)| \to \infty$ as $\lambda \to 0$.

Consequently, the projective variety associated with $V_f$ must be zero-dimensional. By B\'ezout's theorem (see, e.g., \cite[Thm.~18.3]{Harris1992}), the number of complex lines in $V_f$ is finite and bounded by an integer $r \ge 1$ depending only on the degrees of the polynomials defining the image spaces of $\A$. Restricting to real frequencies, this strictly implies that the number of distinct real directions $\zeta \in S^{n-1}$ mapping to $f$ satisfies:
\[
    \# \{\zeta \in S^{n-1} : f \in \im \A(\zeta) \} \le r.
\]
Since complex-ellipticity trivially implies $\R$-ellipticity, this proves (iii).

\medskip

\emph{Step 2: (iii) $\implies$ (ii)} \\
Assume $\A(D)$ is $\R$-elliptic and that there exists $r \ge 1$ bounding the number of directions as in (iii). We want to show that the $(n-1)$-dimensional wave cone is trivial, i.e., $\mathcal{I}^{n-1}_\A = \{0\}$. 

Suppose for contradiction there exists a non-zero vector $w \in \mathcal{I}^{n-1}_\A$. By Lemma~\ref{lem:crofton}, the spherical slice of real frequencies mapping to $w$, $A_w \defeq \{\zeta \in S^{n-1} : w \in \im \A(\zeta)\}$, must contain a non-trivial smooth regular curve. Since a continuous curve contains uncountably many points, this contradicts the finite bound $\# A_w \le r$ established in (iii). Therefore, no such $w$ can exist, yielding $\mathcal{I}^{n-1}_\A = \{0\}$. This proves (ii).

\medskip

\emph{Step 3: (ii) $\implies$ (i)} \\
Assume $\A(D)$ is $\R$-elliptic and vanishing on hyperplanes ($\mathcal{I}^{n-1}_\A = \{0\}$). We prove it must be $\C$-elliptic by contraposition. Assume that $\A(D)$ is real-elliptic but fails to be complex-elliptic. By definition, there exists a non-zero complex frequency $\xi = x + iy \in \C^n$ and a non-zero vector $a = u + iv \in \C \otimes E$ such that $\A(x+iy)(u+iv) = 0$. 

Because $\A(D)$ is real-elliptic and homogeneous of first-order, the vectors $x,y \in \R^n$ must be linearly independent; otherwise, $\xi$ would be a real scalar multiple, yielding $\A(\xi)u = \A(\xi)v = 0$ on $\R^n \smallsetminus \{0\}$, which violates $\R$-ellipticity. 
Expanding the complex system and separating the real and imaginary components yields:
\[
    \A(x)u - \A(y)v = 0 \quad \text{and} \quad \A(x)v + \A(y)u = 0.
\]
Define $w = \A(x)u = \A(y)v$. Since $x$ and $y$ are linearly independent and $a \neq 0$, the real-ellipticity of $\A(D)$ guarantees that $w \neq 0$. 

Let $\zeta = \alpha x + \beta y$ be an arbitrary non-zero frequency in $\mathrm{span}\{x,y\} \subset \R^n$. Using the identities derived above, we compute the image of the vector $\alpha u + \beta v$:
\begin{align*}
    \A(\alpha x + \beta y)(\alpha u + \beta v) &= \alpha^2 \A(x)u + \alpha\beta\A(x)v + \alpha\beta\A(y)u + \beta^2\A(y)v \\
    &= \alpha^2 w - \alpha\beta\A(y)u + \alpha\beta\A(y)u + \beta^2 w \\
    &= (\alpha^2 + \beta^2)w.
\end{align*}
Since $\zeta \ne 0$ in the real span of $x$ and $y$, we have $\alpha^2 + \beta^2 \neq 0$. Thus, $w \in \im \A(\zeta)$ for every non-zero frequency in the $2$-dimensional real plane spanned by $x$ and $y$. 

Now, let $V \in \mathrm{Gr}(n,n-1)$ be an arbitrary hyperplane. By dimensional counting, $\dim(V \cap \mathrm{span}\{x,y\}) \ge (n-1) + 2 - n = 1$. Consequently, $V$ must contain some non-zero frequency $\zeta_0 \in \mathrm{span}\{x,y\}$. Since $w \in \im \A(\zeta_0)$, it follows that $w \in \bigcup_{\zeta \in V} \im \A(\zeta)$.  Because this holds for every hyperplane $V$, the non-zero vector $w$ belongs to the intersection $\mathcal I^{n-1}_\A$. Therefore, $\mathcal I^{n-1}_\A \neq \{0\}$, which contradicts the assumption of hyperplane vanishing. Thus, $\A(D)$ must be $\C$-elliptic, completing the proof.
\end{proof}

\subsection{Generalization to Higher-Order Operators}\label{sec:higher}

It is a natural question to ask whether $\R$-ellipticity and hyperplane vanishing, as a characterization of $\C$-ellipticity, extends to general homogeneous operators of arbitrary order. By employing the exact order-reduction framework recently utilized in~\cite{AS2025}, we can show that it is a strictly necessary condition (but generally not a sufficient one).

Let $\A(D)$ be a $k$-th order homogeneous differential operator acting on $E$. Its first-order reduction is an operator $\mathbb{L}_\A(D) \defeq \mathbb{L}(D)$ acting on the space of $(k-1)$-symmetric tensors $V = \odot^{k-1}\R^n \otimes E$, mapping into $F \times Y$ (where $Y$ is the target space for the curl-free compatibility conditions). The linearized operator takes the form:
\[
    \mathbb{L}(D)v = (\tilde{\A}(D)v, \operatorname{curl} v),
\]
where $\tilde \A(D)$ is a first-order linear operator. 
In the frequency domain, the principal symbol of $\mathbb{L}$ evaluated at a frequency $\zeta \in \R^n$ acting on a tensor $\hat{v} \in V$ is given by
\[
    \mathbb{L}(\zeta)\hat{v} = (\tilde{\A}(\zeta)\hat{v}, \zeta \wedge \hat{v}).
\]
To understand the relationship between the original higher-order symbol $\A(\zeta)$ and the linearized symbol $\mathbb{L}(\zeta)$, we evaluate the latter on an \emph{exact} symmetric tensor. Choosing $\hat{v} = \zeta^{\otimes (k-1)} \otimes a$ for an arbitrary vector $a \in E$, the curl compatibility condition trivially vanishes ($\zeta \wedge \hat{v} = 0$), and the principal part collapses exactly into the original higher-order symbol:
\[
    \tilde{\A}(\zeta)(\zeta^{\otimes (k-1)} \otimes a) = \A(\zeta)a.
\]
This yields the fundamental algebraic mapping:
\begin{equation}\label{eq:red}
    \mathbb{L}(\zeta)(\zeta^{\otimes (k-1)} \otimes a) = (\A(\zeta)a, 0).
\end{equation}
In particular,
\begin{equation}\label{eq:eq}
\text{$\A(D)$ is $\C$-elliptic \quad $\Leftrightarrow$ \quad  $\mathbb{L}(D)$ is $\C$-elliptic.}
\end{equation}
Identity~\eqref{eq:red} demonstrates that for every frequency $\zeta \in \R^n$, the image of the original $k$-th order symbol is canonically embedded into the image of the first-order symbol:
\begin{equation}\label{eq:eq1}
    \im \A(\zeta) \times \{0\} \subset \im \mathbb{L}(\zeta).
\end{equation}

With this exact algebraic mapping established, the proofs of the higher-order corollaries follow directly by embedding the associated wave cones.

\begin{proof}[Proof of Corollary \ref{cor:high1}]Because~\eqref{eq:eq1} holds point-wise for every $\zeta$ and is preserved under intersections over the Grassmannian, the $\ell$-dimensional image cone of the higher-order operator is canonically embedded into the $\ell$-dimensional one of its order-reduced operator:
\[
    \mathcal I^{n-1}_{\A} \times \{0\} \subset \mathcal I^{n-1}_{\mathbb{L}}.
\]
By Theorem~\ref{thm:main}, the codimension-one wave cone of $\mathbb{L}$ is strictly trivial, meaning $\mathcal I^{n-1}_{\mathbb{L}} = \{(0,0)\}$. The subset inclusion immediately forces
\[
    \mathcal I^{n-1}_{\A} = \{0\},
\]
which demonstrates every $\C$-elliptic operator is hyperplane vanishing.
\end{proof}

\begin{proof}[Proof of Corollary \ref{cor:high}]
Combining Theorem~\ref{thm:main} and~\eqref{eq:eq} directly yields the equivalence stated in Corollary~\ref{cor:high}: $\A(D)$ is $\C$-elliptic if and only if $\mathbb{L}(D)$ is $\R$-elliptic and hyperplane vanishing.
\end{proof}

\subsection{Dimension and trace results}

Once we have established the hyperplane vanishing of complex-elliptic operators, the following arguments formally detail how the triviality of the wave cone translates into the stated dimensional and trace properties.

\begin{proof}[Proof of Corollary \ref{cor:estimates}]
Let $\A(D)$ be a complex-elliptic operator. By Theorem \ref{thm:main}, its $(n-1)$-dimensional wave cone is strictly trivial, meaning $\mathcal I^{n-1}_\A = \{0\}$. Under this exact algebraic condition, Arroyo-Rabasa, De Philippis, Hirsch, and Rindler established in \cite[Cor. 1.4]{ADHR2019} that any measure $\mu = \A u$ representing the differential constraint must satisfy the dimensional estimate $|\mu| \ll \mathscr{I}^{n-1} \ll \mathscr H^{n-1}$. The corollary follows immediately.
\end{proof}

\begin{proof}[Proof of Corollary \ref{cor:traces}]
The existence of interior jump traces on rectifiable sets follows by the very definition of an $\mathscr H^{n-1}$-rectifiable set as a countable union $\{\Gamma_i\}$ of Lipschitz graphs up to an $\mathscr H^{n-1}$-nullset $S$. On each $\Gamma_i$ one simply appeals to~\cite[Thm.~2.4]{AS2025}; the important part is to discard $S$ since $|\A(D) u|(S) = 0$ by Corollary~\ref{cor:estimates}. Further details are left to the reader.
\end{proof}

Building on the existence of these two-sided traces, we can now completely isolate and characterize the lowest-dimensional component of the singular measure.

\begin{proof}[Proof of Corollary \ref{cor:jump}]
Let $\mu = |\A(D)u|$ and denote the set of infinite density points by 
\[
A_\infty \defeq \{x \in \Omega : \theta^{*(n-1)}(\mu, x) = \infty\}.
\]
By standard geometric measure theory density theorems (see, e.g., \cite[Thm.~2.56]{AFP2000}), $\mathscr{H}^{n-1}(A_\infty) = 0$. Since Corollary~\ref{cor:estimates} establishes $\mu \ll \mathscr{H}^{n-1}$, it immediately follows that $\mu(A_\infty) = 0$, and therefore $\A u$ vanishes on $A_\infty$.

Consequently, the jump measure $\A^j u$ is strictly supported on the intermediate set 
\[
    S \defeq \{x \in \Omega : 0 < \theta^{*(n-1)}(\mu, x) < \infty\}.
\]
The set $S$ is $\sigma$-finite with respect to $\mathscr{H}^{n-1}$. By the Besicovitch--Federer projection theorem (see, e.g., \cite[Thm.~18.1]{Mat}), we can decompose $S$ up to an $\mathscr{H}^{n-1}$-null set into a disjoint union $S = \Gamma \cup U$, where $\Gamma$ is countably $\mathscr{H}^{n-1}$-rectifiable and $U$ is purely unrectifiable. Since $S$ is $\sigma$-finite, the pure unrectifiability of $U$ implies that its integral geometric measure strictly vanishes: $\mathscr{I}^{n-1}(U) = 0$. 

Invoking the optimal dimensional estimate $\mu \ll \mathscr{I}^{n-1}$ from Corollary~\ref{cor:estimates} once more, we deduce that $\mu(U) = 0$. This forces the measure to vanish entirely on the unrectifiable part, leaving $\A^j u = \A(D)u \lfloor_\Gamma$. 

Because $\Gamma$ is an $\mathscr{H}^{n-1}$-rectifiable set, we may directly apply Corollary~\ref{cor:traces} to represent the measure over $\Gamma$, yielding
\[
    \A^j u = \A(\nu)(u^+ - u^-) \, \mathscr H^{n-1} \lfloor_\Gamma.
\]
By definition, this density strictly vanishes wherever $u^+ = u^-$ (the approximate continuity points). Thus, the measure naturally concentrates exactly on the classical jump set $J_u \subset \Gamma$, concluding the proof.
\end{proof}

\subsection{Failure of the higher-order characterization} 
We provide a counterexample here showing that $\R$-ellipticity and hyperplane vanishing are insufficient to guarantee $\C$-ellipticity for operators of order higher than one:

\begin{example}\label{ex}
We construct a real-elliptic and hyperplane vanishing operator $\A(D)$ of order two ($k = 2)$  in two dimensions ($n=2$), which fails to be complex-elliptic. 
 
Define $\A(D)$ mapping from $E = \R^2$ to $F = \R^4$ by its symbol matrix:
\[
    \A(\xi) = \begin{pmatrix} \xi_1^2 - \xi_2^2 & 2\xi_1\xi_2 \\ - 2\xi_1\xi_2 & \xi_1^2 - \xi_2^2 \\ \xi_1^2 + \xi_2^2 & 0 \\ 0 & \xi_1^2 + \xi_2^2 \end{pmatrix}.
\]

Clearly, $\A(D)$ is real-elliptic because its lower $2 \times 2$ block is the symbol of the Laplacian operator. 
Now, we complexify the frequency variables to $\zeta = (\zeta_1, \zeta_2) \in \C^2$ and evaluate the symbol at $\zeta = (1, i)$. Substituting $\zeta_1 = 1$ and $\zeta_2 = i$ yields:
\[
    \A(1, i) = \begin{pmatrix} 2 & 2i \\ -2i & 2 \\ 0 & 0 \\ 0 & 0 \end{pmatrix} \quad \implies \quad \A(1, i)\begin{pmatrix} -i\\ 1 \end{pmatrix} = 0.
\]
Therefore, $\A(D)$ is not complex-elliptic.

For $n=2$, the wave cone is the intersection of the images over all non-zero frequencies. Applying double-angle identities, the symbol matrix simplifies to:
\[
    \A\left(\cos \frac{\theta}{2}, \sin \frac{\theta}{2}\right) = \begin{pmatrix} M_\theta  \\
    I_{2 \times 2}\end{pmatrix} \defeq \begin{pmatrix} \cos(\theta) & \sin(\theta) \\ -\sin(\theta) & \cos(\theta) \\ 1 & 0 \\ 0 & 1 \end{pmatrix}.
\]
Suppose a constant vector $W = (W_1, W_2, W_3, W_4)^T \in \R^4$ belongs to the wave cone $\mathcal I^1_{\A}$. In particular $(W_1,W_2)^T \in \bigcap_{M \in \mathrm{SO}(2)} \im M = \{(0,0)\}.$
On the other hand, for every angle $\theta$, there must exist an input vector $u_\theta = (a_\theta, b_\theta)^T \in \R^2$ such that $\A(\cos\theta, \sin\theta)u_\theta = W$.  The bottom two rows of this linear system dictate that $a_\theta \equiv W_3$ and $b_\theta \equiv W_4$. Evaluating at $\theta = 0$, 
\begin{align*}
\begin{pmatrix} 0 \\ 0\end{pmatrix} = \begin{pmatrix} W_1 \\ W_2\end{pmatrix} =   M_0 \begin{pmatrix} W_3 \\ W_4\end{pmatrix} =  \begin{pmatrix} W_3 \\ W_4\end{pmatrix} \quad \Rightarrow \quad W = 0.
\end{align*}
Therefore $\A(D)$ is vanishing on hyperplanes.
\end{example}

\begin{remark}
The failure of the hyperplane-vanishing characterization is related (but not equivalent) to the fact that ellipticity and Van Schaftingen's canceling property are not enough to characterize complex-ellipticity for higher order operators in dimension $n=2$, an observation already noted by Gmeineder and Raita (see~\cite[Prop.~1.2]{GR2019}).
\end{remark}

\section{Plane waves and rigidity}\label{sec:rigidity}

In this section, we establish the explicit plane-wave decomposition of measures constrained by complex-elliptic operators, presenting the proofs of Lemma~\ref{lem:rigidity} and Corollary~\ref{cor:pure_tangents}. The overarching strategy is to transform the PDE constraint into a homogeneous algebraic system where the severe twisting of the complex-elliptic symbol rigidly forces the solutions into purely one-dimensional profiles. 

We rely on the natural differential annihilator (see~\cite{VS2013}) of $\A(D)$, which is given by an operator $\Q(D)$ with constant real coefficients making the following sequence exact:
\[
0 \longrightarrow E \xrightarrow{\A(\xi)} F \xrightarrow{\Q(\xi)} G \longrightarrow 0
\]
for all frequencies $\xi \in \C^n \smallsetminus \{0\}$, where $G$ is a finite dimensional Euclidean space. For a fixed vector $P \in F$, we consider the $G$-valued operator acting on scalar-valued distributions by
\[
	\B(D)\varphi \coloneqq \Q(D)[P\varphi], \qquad \varphi \in \mathscr D'(\Omega). 
\]

By assumption, $u$ satisfies the differential constraint $\A(D) u = P \sigma$ in $\Omega$, where $\sigma$ is a Radon measure. Applying the annihilator yields:
\[
    \B(D)\sigma = \Q(D)[P \sigma] = \Q(D)\A(D)u = 0
\]
in the sense of distributions on $\Omega$. 

The complex algebraic variety associated with the annihilator is
\[
	V_P = \{\xi \in \C^n \mid \B(\xi) = 0\} = \{ \xi \in \C^n \mid P \in \im \A(\xi) \}.
\]
Because $\A(D)$ is complex-elliptic, Theorem~\ref{thm:main} ensures that the projective variety associated with $V_P$ is finite (dimension zero). Thus, the affine variety $V_P$ must either be empty or consist of a finite union of complex lines passing through the origin. Let us assume $V_P = \ell_1 \cup \dots \cup \ell_r$. By B\'ezout's theorem (see, e.g., \cite[Thm.~18.3]{Harris1992}), the degree of this projective variety---and hence the maximum number of distinct complex lines $r$---is strictly bounded by the product of the degrees of the polynomials defining $\B(\xi)$. Consequently, $r$ is an integer depending strictly on the algebraic symbol of $\A$ (independently of $P$). For physical relevance, we assume the lines are generated by real directional vectors $\zeta_j \in S^{n-1}$, that is $\ell_j = \mathrm{span}_\C \{\zeta_j\}$.

We break the proof into three main steps: first, utilizing the classical Ehrenpreis-Palamodov principle to decompose the measure algebraically into distributional components supported on characteristic lines; second, applying a geometric slicing argument to upgrade the constituent transversal profiles from distributions to Radon measures; and finally, using Smith's exact differential inverse to transfer this plane-wave decomposition back onto $u$ and deduce its $BV_{loc}$ regularity.

\begin{proof}[Proof of Lemma \ref{lem:rigidity}] \phantom{.}

\medskip

\emph{Step 1: Algebraic decomposition via the Ehrenpreis-Palamodov principle.} \\
Let $J = \langle \B_1(\zeta), \dots, \B_O(\zeta) \rangle \subset \C[\zeta]$ be the ideal generated by the rows of $\mathbb B$. By construction, the algebraic variety (the common zero set) associated with $J$ is the union of distinct lines $V_P = \ell_1 \cup \dots \cup \ell_r$. 
By the Lasker-Noether theorem, $J$ admits a minimal \emph{primary decomposition} $J = \bigcap_{j=1}^r \mathfrak{q}_j$. This decomposition admits a direct PDE interpretation: each \emph{primary ideal} $\mathfrak{q}_j$ corresponds to a sub-system of PDEs supported on a single irreducible geometric component of the variety---in this case, the line $\ell_j$. The \emph{radical} of $\mathfrak{q}_j$, denoted $\sqrt{\mathfrak{q}_j}$, is the \emph{prime ideal} $\mathfrak{p}_j = I(\ell_j)$ consisting of all polynomials that strictly vanish on $\ell_j$. 

Since $\Omega$ is a convex set of $\R^n$, in particular it is $\mathbf B$-convex in the sense of Palamodov (see~\cite[Ch. VII \S8]{Palamodov1970}). Hence, we may apply the Ehrenpreis--Palamodov representation principle (see~\cite[Ch. VI, Thm.~1]{Palamodov1970}) to discover that $\sigma$ splits uniquely into a finite sum of distributions
\[
    \sigma = \sum_{j=1}^r T_{j},
\]
where each component $T_{j} \in \mathscr D'(\Omega)$ satisfies $q(D)T_{j} = 0$ for every polynomial $q \in \mathfrak{q}_j$. 
By Hilbert's Nullstellensatz, there exists an integer $N \ge 0$ such that $\mathfrak{p}_j^N \subset \mathfrak{q}_j$, for each $j = 1,\dots,r$. Since $\mathfrak{p}_j$ is generated by the transverse directions $\{v \cdot \zeta \mid v \in \ell_j^\perp\}$, it follows that $(v \cdot D)^N T_{j} = 0$. This dictates that $T_{j}$ must be a polynomial of degree at most $N-1$ in the directions transverse to $\ell_j$, with distributional one-dimensional profiles along $\ell_j$. 

More precisely, relying on the rotated domain $\Omega_{\zeta_j} \defeq R_{\zeta_j}^{-1}(\Omega)$ introduced in Definition~\ref{def:plane_waves}, we can express each term as
\[
    T_{j} = \sum_{0 \le |\alpha| < N} (R_{\zeta_j})_\# \big(g_{j,\alpha} \otimes (y^\alpha \mathscr L^{n-1})\big) \lfloor_\Omega,
\]
where $y$ denotes the transversal variables in $\R^{n-1}$, $\alpha$ is a multi-index, and each $g_{j,\alpha}$ is a 1D distribution defined over the 1D projection of $\Omega_{\zeta_j}$. Moreover, the decomposition above is unique. 

\begin{example}[Transversal components]
    To briefly illustrate the distinction between $\mathfrak{q}_j$ and its radical $\mathfrak{p}_j$, consider a smooth function $f$ on $\R^2$ satisfying $\partial_{x_2}^2 f = 0$. The associated ideal is $J = \langle \zeta_2^2 \rangle \subset \C[\zeta_1, \zeta_2]$, which is primary ($\mathfrak{q} = \langle \zeta_2^2 \rangle$). Its geometric zero set is the $\zeta_1$-axis, $\ell = \{\zeta_2 = 0\}$, and its prime ideal is $\mathfrak{p} = \langle \zeta_2 \rangle$. Analytically, the PDE associated to $\mathfrak{p}$ is $\partial_{x_2} f = 0$, yielding purely longitudinal 1D plane waves $f(x_1, x_2) = g(x_1)$. The PDE associated to $\mathfrak{q}$ is $\partial_{x_2}^2 f = 0$, whose solutions are $f(x_1, x_2) = g(x_1) + x_2 h(x_1)$.
\end{example} 

\emph{Step 2: Upgrading distributions to measures via transversal slicing.} \\
We now establish that the 1D distributions $g_{j,\alpha}$ are in fact Radon measures. Fix an index $j_0$ (without loss of generality, $j_0 = 1$). We will isolate the distributions on $\ell_1$ using a local slicing argument.

Let $Q \subset \Omega$ be a small cube centered at $x_0$ chosen such that its rotated counterpart $Q_{\zeta_1} = R_{\zeta_1}^{-1}(Q) \subset \Omega_{\zeta_1}$ has sides parallel to the split coordinate axes $\R \times \R^{n-1}$. Choose a family of compactly supported, smooth test functions $\{\psi_\beta\} \subset C_c^\infty(\R^{n-1})$ localized strictly within the transversal projection of $Q_{\zeta_1}$ such that for all multi-indices $0 \le |\alpha|, |\beta| < N$ it holds
\begin{equation}
    \int_{\R^{n-1}} y^\alpha \psi_\beta(y) \, dy = \delta_{\alpha, \beta}.
\end{equation}
For each $\beta$, consider the map $S_\beta:  \mathscr D'(Q) \to \mathscr D'(Q \cap \{\ell_1 + x_0\})$ given by 
\[
    S_\beta T(\varphi) = T(\psi_\beta \otimes \varphi), \qquad \varphi \in C_c^\infty(Q \cap \{\ell_1 + x_0\}).
\] 
We test the distribution $\sigma$ locally against the transverse profiles $\psi_\beta(y)$. Because $\sigma$ is a signed Radon measure on $\Omega$, the transversal integration against a smooth test function rigidly projects it onto a signed Radon measure on the longitudinal line segment of $Q_{\zeta_1}$. Applying this projection to our decomposition of $\sigma$ yields:
\begin{equation}
    \mathcal M_{loc}(\R) \ni  S_\beta\sigma = g_{1,\beta} + \sum_{j=2}^r \sum_{|\alpha| < N} S_\beta \left( (R_{\zeta_j})_\# \big(g_{j,\alpha} \otimes (y_{\zeta_j}^\alpha \mathscr L^{n-1})\big)  \right).
\end{equation}

Let us fix $j \ge 2$, the line $\ell_j$ is strictly transverse to $\ell_1$. Geometrically, the directional vectors $\zeta_1$ and $\zeta_j$ are linearly independent. To determine the regularity of the projection, we evaluate the action of $S_\beta$ on the transverse distribution term $T_j = (R_{\zeta_j})_\# (g_{j,\alpha} \otimes (y_{\zeta_j}^\alpha \mathscr L^{n-1}))$ against a longitudinal test function $\varphi \in C_c^\infty(\R)$ localized on the 1D projection of $Q_{\zeta_1}$.

By definition, this pairing lifts to evaluating $T_j$ against the full test function $\Phi$ on $Q$, where $R_{\zeta_1}^*\Phi(t, y_{\zeta_1}) = \varphi(t)\psi_\beta(y_{\zeta_1})$ with $t = x \cdot \zeta_1$. Transforming the integral into a local linear coordinate system $(t, s, z) \in \R \times \R \times \R^{n-2}$ perfectly aligned with both characteristic directions—such that $t = x \cdot \zeta_1$ and $s = x \cdot \zeta_j$—the transversal polynomial $y_{\zeta_j}^\alpha$ maps to a smooth polynomial $P(t, s, z)$, and the transversal weight $\psi_\beta$ maps to a smooth, compactly supported function $\tilde{\psi}_\beta(t, s, z)$.

The distributional pairing can therefore be expressed as:
\begin{equation*}
    \langle S_\beta T_j, \varphi \rangle = \left\langle g_{j,\alpha}(s), \int_{\R^{n-1}} P(t, s, z) \tilde{\psi}_\beta(t, s, z) \varphi(t) \, dt \, dz \right\rangle_{s},
\end{equation*}
where the constant Jacobian of the coordinate change has been absorbed into $P$. Because we are integrating out the $n-2$ auxiliary variables $z$, we can define the \emph{smooth} partial integral
\begin{equation*}
    K(t, s) = \int_{\R^{n-2}} P(t, s, z) \tilde{\psi}_\beta(t, s, z) \, dz.
\end{equation*}
Since $\tilde{\psi}_\beta$ is smooth and compactly supported, $K(t, s)$ is strictly a $C_c^\infty(\R^2)$ function. The standard properties of distributions mapping smooth parametric test functions allow us to rewrite the pairing as:
\begin{equation*}
    \langle S_\beta T_j, \varphi \rangle = \int_\R \langle g_{j,\alpha}(\cdot), K(t, \cdot) \rangle \, \varphi(t) \, dt.
\end{equation*}
By the classical theorem on distributions acting on smooth parametric families of test functions (see, e.g., H\"ormander~\cite[Thm.~2.1.3]{hormander2003}), the parametric evaluation $t \mapsto \langle g_{j,\alpha}(\cdot), K(t, \cdot) \rangle$ strictly defines a $C^\infty$ function on the longitudinal line segment. Consequently, projecting a distribution supported along the transverse line $\ell_j$ unconditionally smooths out its singularities.

Thus, on the 1D domain, we have:
\begin{equation}
    g_{1,\beta} = S_\beta\sigma - \text{(smooth function)}.
\end{equation}
Since $S_\beta\sigma$ is a Radon measure and a smooth function acts as an absolutely continuous Radon measure, their difference $g_{1,\beta}$ must itself be a Radon measure. By symmetry, $g_{j,\alpha}$ is a Radon measure for all $j$ and $\alpha$.

Because each $g_{j,\alpha}$ is a one-dimensional Radon measure on $\R$, its cumulative distribution function $w_{j,\alpha}(s) = g_{j,\alpha}((-\infty, s))$ is a function of bounded variation $BV_{loc}(\R)$, and we can write $g_{j,\alpha} = w_{j,\alpha}'$ in the distributional sense. 
This establishes the rigorous representation 
\[
    \sigma = \sum_{j=1}^r \sum_{|\alpha| < N} (R_{\zeta_j})_\# \big(w_{j,\alpha}' \otimes (p_{j,\alpha} \mathscr L^{n-1})\big) \lfloor_\Omega.
\]

\emph{Step 3: Plane-wave decomposition and regularity of $u$.} \\
Because $\A(D)$ is $\C$-elliptic, its symbol $\A(\xi)$ possesses no non-trivial zeros in $\C^n$. By an application of Hilbert's Nullstellensatz exactly as shown by Smith \cite[p.~57]{Smith70}, the ideal generated by the matrix entries of $\A(\xi)$ must contain all homogeneous polynomials of some sufficiently high degree. Therefore, there exists an integer $m \ge 1$---dictated entirely by the algebraic properties (degree and coefficients) of $\A$---and a homogeneous constant-coefficient differential operator $H(D)$ such that the operator identity holds:
\[
    \nabla^m = H(D)\A(D).
\]
Applying this differential identity to our constraint $\A(D)u = P\sigma$ yields $\nabla^m u = H(D)[P \sigma]$. 
Since $\sigma$ decomposes entirely into a finite sum of transversally polynomial plane waves along the characteristic directions $\zeta_j$, applying the differential operator $H(D)P$ to $\sigma$ produces another sum of similar plane wave distributions precisely along the same directions. Because every partial derivative of $u$ of order $m$ is then also of this form, $u$ itself must also decompose into a sum of transversally polynomial plane waves along these directions (albeit the profiles may not be measures), up to an additive polynomial of degree at most $m-1$. 

Thus, we may express $u$ exactly as:
\[
    u = Q_{m-1}(x) + \sum_{j=1}^r u_j, \qquad u_j = \sum_{\beta} (R_{\zeta_j})_\# \big(v_{j,\beta} \otimes (q_{j,\beta} \mathscr L^{n-1})\big) \lfloor_\Omega,
\]
where $Q_{m-1}$ is a polynomial of degree at most $m-1$ and $v_{j,\beta}$ are 1D distributional profiles taking values in $E$. 

To establish the $BV_{loc}$ regularity of these profiles, we isolate the components along each characteristic line $\ell_j$ (as done in Step 2 via algebraic independence of the polynomials), yielding the localized equality $\A(D)u_j = P\sigma_j$. Under the coordinate splitting $x = R_{\zeta_j}(t,y)$, the operator decomposes geometrically as $\A(D) = \A(\zeta_j)\partial_t + \A_{\perp j}(\nabla_y)$. 

Let $N_j$ be the maximum transversal degree inside the plane-wave $u_j$. The highest transversal degree resulting from applying $\A(D)$ to $u_j$ is generated exclusively by the longitudinal derivative $\partial_t$ (since the transversal derivatives $\nabla_y$ necessarily lower the polynomial degree):
\[
    \A(D) u_j = \sum_{|\beta|=N_j} (R_{\zeta_j})_\# \big(\A(\zeta_j) v'_{j,\beta} \otimes (q_{j,\beta}\mathscr L^{n-1})\big) + (\text{lower transversal degree}).
\]
Equating this to $P\sigma_j$ and isolating the terms of maximal degree $N_j$ through the linear independence of polynomials, we obtain the system of equations:
\[
    \A(\zeta_j) v'_{j,\beta} = P g_{j,\beta},
\]
for the associated measures $g_{j,\beta}$ with $|\beta|=N_j$ arising from the decomposition of $\sigma_j$.

Fix $\beta$ such a multi-index with $|\beta| = N_j$. By hypothesis, $P \in \im \A(\zeta_j)$. Because $\A(D)$ is $\C$-elliptic, the matrix $\A(\zeta_j)$ is strictly injective. Consequently, there exists a unique vector $B_j \in E$ such that $\A(\zeta_j)B_j = P$. This forces the distributional relation $v'_{j,\beta} = B_j g_{j,\beta}$. Since $g_{j,\beta}$ is a Radon measure, its anti-derivative $v_{j,\beta}$ strictly belongs to $BV_{loc}(\R, E)$.

We proceed by downward induction on the transversal degree $k$ from $N_j$ down to $0$. Suppose we have determined that the profiles $v_{j,\gamma}$ for $|\gamma| > k$ belong strictly to $BV_{loc}(\R, E)$. To solve for the profiles with $|\beta| = k$, we subtract the higher-degree terms from both sides of the localized equality. Because the previously determined profiles are $BV_{loc}$, their transversal derivatives $\A_{\perp j}(\nabla_y)$ yield lower-degree polynomial plane waves whose coefficients are exactly Radon measures (arising from first-order distributional derivatives of $BV_{loc}$ profiles). By linear independence of the transversal polynomials, projecting the remaining equation onto the degree-$k$ components yields a system of the form $\A(\zeta_j)v'_{j,\beta} = \tilde{g}_{j,\beta}$, where $\tilde{g}_{j,\beta}$ is an $F$-valued Radon measure strictly residing in $\im \A(\zeta_j)$. The strict injectivity of $\A(\zeta_j)$ allows us to algebraically invert this to $v'_{j,\beta} = \A(\zeta_j)^{-1}\tilde{g}_{j,\beta}$, deducing that every individual profile is the anti-derivative of a measure. 

Therefore, $u$ is fully described by a finite sum of transversally polynomial plane waves possessing purely $BV_{loc}(\R,E)$ profiles. 
\end{proof}

While the preceding lemma establishes that the local structure of the measure is governed by polynomial plane waves, these transversal polynomials are essentially artifacts of the differential constraints' higher-order compatibilities. The following proof demonstrates that an iterated blow-up process systematically purges these polynomial variations, ultimately isolating a pure plane wave.

\begin{proof}[Proof of Corollary \ref{cor:pure_tangents}]
Let $x \in \Omega$ be a point such that the set of tangents $\Tan_{BV^\A}(u, x)$ is non-empty, and let $v \in \Tan_{BV^\A}(u, x)$. A foundational principle of geometric measure theory (see, e.g., Preiss~\cite{Preiss1987}) establishes that tangent maps of tangent maps remain tangent maps of the original function. Therefore, any tangent of $v$ belongs to $\Tan_{BV^\A}(u, x)$.

We distinguish two cases based on the structure of $v$:\\

\emph{Case 1: $\A v$ possesses a regular point.} If $v$ is not purely singular, there exists a Lebesgue point $y_0$ of $v$. Performing a secondary blow-up of $v$ specifically centered at $y_0$:
\[
    \tilde{v}_{\rho}(y) = \tilde{c}_{y_0, \rho} \left( v(y_0 + \rho y) - \fint_{B_1(0)} v(y_0 + \rho z) \, dz \right)
\]
will weakly-* converge to a constant tangent map $\tilde{v}(y) \equiv c$. A constant map trivially qualifies as a pure plane wave $g(y \cdot \zeta_0)$ for any characteristic direction $\zeta_0$.\\

\emph{Case 2: $\A v$ is purely singular.} 
By the previous discussion, $v \in BV_{loc}(\R^n, E)$ and the associated tangent measure $\mu = \A(D)v$ takes the form $\mu = P_0 \lambda$ for some positive scalar measure $\lambda$ and a fixed vector $P_0 \in F$. By Lemma~\ref{lem:rigidity}, the measure $\mu$ decomposes exactly as a finite sum of characteristic components $\mu = \sum_{j=1}^r \mu_j$, where each $\mu_j$ is supported on families of hyperplanes orthogonal to the distinct characteristic directions $\zeta_j$. Let $S_j$ be a full $\mu_j$-measure set for each $j = 1,\dots,r$. Because the directions $\zeta_j$ are pairwise distinct, the intersection of any two distinct $S_i \cap S_j$ has $\mathscr H^{n-1}$-measure zero. In particular, the intersection $\bigcup_{i \neq j} (S_i \cap S_j)$ has $\mathscr{H}^{n-1}$-measure zero. 

Crucially, by Corollary~\ref{cor:estimates}, the measure $\mu = \A(D)v$ is absolutely continuous with respect to $\mathscr{H}^{n-1}$ (since $\A$ is complex-elliptic). Therefore, $\mu$ assigns strictly zero mass to the intersection of these singular supports:
\[
    |\mu|\left( \bigcup_{i \neq j} (S_i \cap S_j) \right) = 0.
\]
Since $\A v$ is non-trivial, $|\mu| > 0$. Thus, there must exist a Lebesgue point $y_0 \in B_1$ of $\A v$ that belongs to exactly \emph{one} $S_j$. Without any loss of generality, we may assume $j = 1$ so that $y_0$ is also a Lebesgue density point of $\A v \llcorner S_1$ where the tangent cone is well defined, and hence
\[
    \Tan(\mu_1, y_0) = \Tan(\A u, y_0) \subset \Tan(\A u, x).
\]
By construction, $\mu_1$ is a transversally polynomial plane wave measure. However, as we blow up at the fixed center $y_0$, the spatial coordinates scale as $y_0 + \rho y$. In the infinitesimal limit $\rho \to 0^+$, any transverse polynomial unconditionally evaluates to a constant scalar $q_\beta(y_{0})$. This geometric collapse forces the blow-up $w$ of $v$ at $y_0$ to lose all transverse spatial dependence, yielding a pure plane wave measure: 
\[
    \A w(dy) = P_0 g(y \cdot \zeta_1).
\]
for some measure $g$ on $[-1,1]$. If $g$ has a non-trivial regular point, we can simply argue as in Case 1. We may hence assume that $\A w$ is purely singular. 

Arguing similarly as in the proof of Lemma~\ref{lem:rigidity}, we discover that 
\[
    w(y) = Q_{m-1}(y) + b_0 h(y \cdot \zeta_1)
\]
where $Q_{m-1}$ is an $\A$-free polynomial of degree less than or equal to $m-1$ and $\A(\zeta_1)b_{0} = P_0$. Since $\A w$ is singular, the smooth additive polynomial $Q_{m-1}$ must be strictly constant (otherwise $\A w$ would possess a regular point), confirming $w$ is exactly a pure plane wave map.
\end{proof}

\section{Examples and characteristic directions}\label{sec:examples}
To provide concrete illustrations of the rigidity phenomena dictated by Lemma~\ref{lem:rigidity}, we evaluate the image cones, the characteristic bound $r$, and the specific characteristic directions $\zeta_j$ for several canonical first-order $\C$-elliptic operators commonly studied in the calculus of variations (see, e.g., \cite{AS2025} for further related geometric examples). 

In what follows, we fix a non-zero polar vector $P \in \mathcal I_\A$. In light of Theorem~\ref{thm:main}, we may restrict our study exclusively to the real algebraic variety:
\[
    V^\R_P = \{ \zeta \in \R^n \setminus \{0\} : P \in \im \A(\zeta) \}.
\]
The integer $r$ exactly dictates the maximal number of distinct real lines contained in $V^\R_P$, which directly correspond to the active characteristic directions bounding the plane-wave expansions.

\subsection{The Gradient Operator}
Consider the standard gradient operator on vector-valued fields $u : \R^n \to \R^m$. Here, $E = \R^m$, $F = \R^{m \times n}$, and $\A(D)u = Du$. 
The principal symbol acting on a real vector $a \in \R^m$ with frequency $\zeta \in \R^n$ is given by
\[
    \A(\zeta)a = a \otimes \zeta.
\]
The operator is strictly $\C$-elliptic, since $a \otimes \zeta = 0$ implies either $a = 0$ or $\zeta = 0$. By definition, the image cone $\mathcal I_D$ consists precisely of all rank-1 matrices $a \otimes \zeta_0 \in \R^{m \times n}$.

Let $P = a \otimes \zeta_0 \in \mathcal I_D$ with $a \neq 0$ and $\zeta_0 \neq 0$. To find the characteristic lines, we solve $P \in \im \A(\zeta)$ for real vectors:
\[
    a \otimes \zeta_0 = b \otimes \zeta \quad \text{for some } b \in \R^m.
\]
To isolate $\zeta$, we can contract this matrix equation. Let $w \in \R^m$ be an arbitrary vector such that $w \cdot a = 1$. Contracting on the left yields
\[
    \zeta_0 = (w \cdot b)\zeta.
\]
Since $\zeta_0 \neq 0$, it must be that $w \cdot b \neq 0$. This immediately forces $\zeta = \frac{1}{w \cdot b}\zeta_0$. Consequently, any real frequency $\zeta \in V^\R_P$ is strictly a scalar multiple of $\zeta_0$. The real variety $V^\R_P$ thus contains exactly one line, $\ell_1 = \mathrm{span}_\R \{\zeta_0\}$. 

We conclude $r = 1$, and the single characteristic direction is $\zeta_0$. By Lemma~\ref{lem:rigidity}, this elegantly recovers the classical rank-one rigidity: $BV$-gradient measures locally reduce to a single pure 1D plane wave.

\subsection{The Symmetric Gradient}
Consider the symmetric gradient operator $\mathcal{E}$ acting on vector fields $u : \R^n \to \R^n$, defined by $\mathcal{E}u = \frac{1}{2}(Du + Du^T)$. Here $E = \R^n$, $F = \R_{\mathrm{sym}}^{n \times n}$, and the symbol acts as
\[
    \A(\zeta)b = b \odot \zeta \defeq \frac{1}{2}(b \otimes \zeta + \zeta \otimes b).
\]
The symmetric gradient is a classical $\C$-elliptic operator. The image cone $\mathcal I_{\mathcal E}$ is exactly the set of all real symmetric matrices of the form $a \odot \zeta_0$, which geometrically have rank at most $2$.

Let $P = a \odot \zeta_0 \in \mathcal I_{\mathcal{E}}$. If $P$ has rank 1 (meaning $a$ is parallel to $\zeta_0$), the computation collapses identically to the standard gradient, yielding $r=1$. Let us assume $P$ has rank 2, meaning $a$ and $\zeta_0$ are linearly independent. We must solve:
\[
    a \odot \zeta_0 = b \odot \zeta \quad \text{for some } b \in \R^n.
\]
Evaluating this symmetric matrix equality against an arbitrary vector $x \in \{a, \zeta_0\}^\perp \subset \R^n$ yields
\[
    0 = \frac{1}{2}\big((b \cdot x)\zeta + (\zeta \cdot x)b\big).
\]
Since $\zeta, b \neq 0$, this orthogonality condition holding for all $x \in \{a, \zeta_0\}^\perp$ forces both $\zeta$ and $b$ to reside entirely within the orthogonal complement of $\{a, \zeta_0\}^\perp$, which is the real plane $\mathrm{span}_\R\{a, \zeta_0\}$. We may therefore uniquely write $\zeta = \alpha a + \beta \zeta_0$ and $b = \gamma a + \delta \zeta_0$ for some real scalars $\alpha, \beta, \gamma, \delta$.

Substituting this expansion back into the symbol equation gives:
\[
    a \odot \zeta_0 = \alpha\gamma (a \otimes a) + \delta\beta (\zeta_0 \otimes \zeta_0) + (\alpha\delta + \beta\gamma)(a \odot \zeta_0).
\]
Because the matrices $\{a \otimes a, \zeta_0 \otimes \zeta_0, a \odot \zeta_0\}$ form a linearly independent basis for the symmetric tensors over this 2D subspace, we equate coefficients to obtain the nonlinear system:
\begin{enumerate}
    \item $\alpha\gamma = 0$
    \item $\delta\beta = 0$
    \item $\alpha\delta + \beta\gamma = 1$
\end{enumerate}
From (1), either $\alpha = 0$ or $\gamma = 0$. 
If $\alpha = 0$, then (3) forces $\beta\gamma = 1$, implying $\beta \neq 0$. By (2), $\delta\beta = 0$ forces $\delta = 0$. This yields $\zeta = \beta \zeta_0$ (and $b = \gamma a$), recovering the real line $\ell_1 = \mathrm{span}_\R\{\zeta_0\}$.
If $\gamma = 0$, an identical argument forces $\beta = 0$, yielding $\zeta = \alpha a$ (and $b = \delta \zeta_0$), recovering a second line $\ell_2 = \mathrm{span}_\R\{a\}$.

Thus, exactly two distinct real lines solve the characteristic equation, so $r = 2$. The characteristic directions driving the plane-wave measures are identically the normalized vectors $\frac{\zeta_0}{|\zeta_0|}$ and $\frac{a}{|a|}$.

\subsection{The Deviatoric Symmetric Gradient}
Consider the trace-free (deviatoric) symmetric gradient $\mathcal{E}^D u = \mathcal{E}u - \frac{1}{n}(\operatorname{div} u) I_n$, acting on $u : \R^n \to \R^n$. The symbol is given by
\[
    \A(\zeta)b = b \odot \zeta - \frac{1}{n}(b \cdot \zeta) I_n.
\]
This operator is strictly $\C$-elliptic for dimensions $n \ge 3$. (For $n=2$, the deviatoric symmetric gradient is algebraically equivalent to the Cauchy-Riemann equations, dropping rank for non-zero complex frequencies $\zeta_1 = \pm i \zeta_2$, and thus fails to be $\C$-elliptic).

Suppose $P = a \odot \zeta_0 - \frac{1}{n}(a \cdot \zeta_0) I_n \in \mathcal{I}_{\mathcal{E}^D}$ with $P \neq 0$. Setting $P = \A(\zeta)b$, we can rewrite the structural equation as
\[
    a \odot \zeta_0 - b \odot \zeta = \lambda I_n, \qquad \text{where } \lambda = \frac{1}{n}(a \cdot \zeta_0 - b \cdot \zeta).
\]
Rearranging this yields $a \odot \zeta_0 - \lambda I_n = b \odot \zeta$. We can rigorously determine $\lambda$ by analyzing the eigenvalue signature of this matrix equality over the reals. 

Because $a$ and $\zeta_0$ are real, linearly independent vectors, the rank-2 symmetric matrix $a \odot \zeta_0$ possesses exactly one strictly positive eigenvalue $\mu_1 > 0$, one strictly negative eigenvalue $\mu_2 < 0$, and the eigenvalue $0$ with multiplicity $n-2$. Consequently, the left-hand matrix $M \defeq a \odot \zeta_0 - \lambda I_n$ has eigenvalues $\mu_1 - \lambda$, $\mu_2 - \lambda$, and $-\lambda$ (the latter with multiplicity $n-2$). 

On the right-hand side, because $b, \zeta \in \R^n$ are also real vectors, the symmetric tensor $b \odot \zeta$ must similarly possess exactly one positive eigenvalue, one negative eigenvalue, and $n-2$ zeros (assuming generic rank 2). Thus, $M$ is algebraically forced to match this exact signature. 

Since $n \ge 3$, the eigenvalue $-\lambda$ appears at least once. 
If we assume $\lambda > 0$, then $M$ possesses at least two strictly negative eigenvalues: $\mu_2 - \lambda < 0$ and $-\lambda < 0$. This contradicts the signature of $b \odot \zeta$. 
Conversely, if we assume $\lambda < 0$, then $M$ possesses at least two strictly positive eigenvalues: $\mu_1 - \lambda > 0$ and $-\lambda > 0$, which is again a contradiction. 

This strict signature mismatch dictates that $\lambda$ must exactly vanish. With $\lambda = 0$, the trace-free condition forcefully degrades the system back to the unconstrained symmetric gradient equality:
\[
    a \odot \zeta_0 = b \odot \zeta.
\]
Consequently, we find that $r = 2$ universally for all dimensions $n \ge 3$, and the underlying characteristic directions governing the plane wave decomposition remain robustly coupled to the constituent vectors parallel to $a$ and $\zeta_0$.

\subsection{The Symmetric $k$-Tensor Gradient}
As a natural generalization of the symmetric gradient, consider the symmetric $k$-tensor gradient $\nabla^s$ acting on fully symmetric $k$-tensor fields $u \in \odot^k \R^n$. Here $E = \odot^k \R^n$, $F = \odot^{k+1} \R^n$. The symbol is given by the symmetric tensor product:
\[
    \A(\zeta)a = a \odot \zeta,
\]
which is strictly $\C$-elliptic since $a \odot \zeta = 0$ implies $a = 0$ or $\zeta = 0$.

Suppose $P = a \odot \zeta_0 \in \mathcal I_\A$ is a fixed polar. The characteristic lines correspond to real frequencies $\zeta \in \R^n$ satisfying
\[
    a \odot \zeta_0 = b \odot \zeta \quad \text{for some } b \in \odot^k \R^n.
\]
By evaluating this fully symmetric $(k+1)$-tensor on a single generic real vector $x \in \R^n$, we obtain an equality of real homogeneous polynomials of degree $k+1$:
\[
    a(x)(\zeta_0 \cdot x) = b(x)(\zeta \cdot x).
\]
Since the real polynomial ring $\R[x_1, \dots, x_n]$ is a unique factorization domain, the real linear form $(\zeta \cdot x)$ must strictly divide the polynomial product $a(x)(\zeta_0 \cdot x)$. As $a(x)$ is a real homogeneous polynomial of degree $k$, it can be factored into at most $k$ irreducible real linear forms (some roots may be complex conjugates, reducing the number of real linear factors). 

Consequently, the linear form $(\zeta \cdot x)$ must either be proportional to $(\zeta_0 \cdot x)$ or to one of the real linear factors embedded within $a(x)$. This strongly forces $\zeta$ to be proportional to $\zeta_0$ or to one of those at most $k$ directions. Therefore, the algebraic variety $V^\R_P$ contains at most $k+1$ real lines.

We conclude that $r \le k+1$ (with equality holding generically when $a$ splits into $k$ distinct real linear factors, none of which align with $\zeta_0$). This example illustrates that the number of characteristic directions $r$ has no universal dimensional bound and can grow arbitrarily large depending entirely on the tensorial order of the operator.

\subsection{A Polynomial Multiplication Operator}
Consider the operator mapping $u : \R^n \to \R^N$ to a vector field in $\R^{n+N-1}$ given by the convolutions of their components:
\[
    (\A(D)u)_m = \sum_{j+l=m+1} \partial_j u_l, \qquad m = 1, \dots, n+N-1.
\]
This operator, discussed as Example 19 in \cite{AS2025}, elegantly translates PDE symbols directly into univariate polynomial algebra. To see this, we identify vectors $\zeta \in \R^n$ and $b \in \R^N$ with the real coefficients of auxiliary polynomials in a dummy variable $t$:
\[
    P_\zeta(t) = \sum_{j=1}^n \zeta_j t^{j-1}, \qquad P_b(t) = \sum_{l=1}^N b_l t^{l-1}.
\]
Under this isomorphism, the symbol $\A(\zeta)b$ is exactly the standard multiplication of these two polynomials:
\[
    P_{\A(\zeta)b}(t) = P_\zeta(t) P_b(t).
\]
The operator is unconditionally $\C$-elliptic because the polynomial ring $\C[t]$ is an integral domain; $P_\zeta(t)P_b(t) = 0$ implies either $P_\zeta = 0$ or $P_b = 0$. 

If we fix a polar $P_C \in \mathcal I_\A$ generated by real polynomials $P_C(t) = P_{\zeta_0}(t)P_a(t)$, the characteristic lines correspond to solving:
\[
    P_C(t) = P_\zeta(t) P_b(t) \quad \text{for real coefficients } \zeta \in \R^n, b \in \R^N.
\]
Finding a valid real frequency vector $\zeta \in \R^n \setminus \{0\}$ is algebraically equivalent to selecting a strictly real polynomial factor $P_\zeta(t)$ of degree exactly $n-1$ that divides $P_C(t)$. 

Since $P_C(t)$ has degree at most $n+N-2$, constructing a real divisor $P_\zeta(t)$ strictly amounts to choosing a subset of roots from $P_C(t)$—where any complex roots must be selected in conjugate pairs to ensure the coefficients $\zeta$ remain real—such that their combined degree is exactly $n-1$. Therefore, the number of distinct real lines $\ell = \mathrm{span}_\R\{\zeta\}$ solving the equation is strictly bounded by the binomial coefficient for choosing these real components:
\[
    r \le \binom{n+N-2}{n-1}.
\]

\end{document}